\documentclass[a4paper,11pt,english]{smfart}
\usepackage[cp1250]{inputenc}
\usepackage{amsmath,amssymb,dsfont}
\usepackage{amsthm}
\usepackage{graphicx}
\usepackage{epstopdf}
\usepackage{xcolor}
\usepackage{multirow}
\usepackage{tabularx}
\usepackage{mathrsfs}
\usepackage{a4wide}

\def\di{\displaystyle}
\def\N{\mathbb{N}}
\def\R{\mathbb{R}}

\def\L{\mathcal{L}}

\def\CC{\mathscr{C}}

\newtheorem{definition}{Definition}
\newtheorem{lemma}{Lemma}
\newtheorem{theorem}{Theorem}
\newtheorem{remark}{\textbf{Remark}}

\newtheorem{corollary}{Corollary}

\begin{document}
\title[Fractional Noether theorem]{About the Noether's theorem for fractional Lagrangian systems and a generalization of the classical Jost method of proof}

\author{Jacky Cresson and Anna Szafra\'{n}ska}
\date{}
\maketitle
\begin{abstract}
Recently, the fractional Noether's theorem derived by G. Frederico and D.F.M. Torres in \cite{FT2} was proved to be wrong by R.A.C. Ferreira and A.B. Malinowska in (see \cite{FM}) using a counterexample and doubts are stated about the validity of other Noether's type Theorem, in particular (\cite{FT},Theorem 32). However, the counterexample does not explain why and where the proof given in \cite{FT2} does not work. In this paper, we make a detailed analysis of the proof proposed by G. Frederico and D.F.M. Torres in \cite{FT} which is based on a fractional generalization of a method proposed by J. Jost and X.Li-Jost in the classical case. This method is also used in \cite{FT2}. We first detail this method and then its fractional version. Several points leading to difficulties are put in evidence, in particular the definition of variational symmetries and some properties of local group of transformations in the fractional case. These difficulties arise in several generalization of the Jost's method, in particular in the discrete setting. We then derive a fractional Noether's Theorem following this strategy, correcting the initial statement of Frederico and Torres in \cite{FT} and obtaining an alternative proof of the main result of Atanackovic and al. \cite{ata}.  
\end{abstract}

\noindent



\tableofcontents

\section{Introduction}

In (\cite{FT},\cite{FT2}), G. Frederico and D.F.M. Torres have formulated a Noether's Theorem for fractional Lagrangian systems. In \cite{FM}, R.A.C. Ferreira and A.B. Malinowska give a counterexample to the main result of \cite{FT2} indicating that they have doubt about (\cite{FT},Theorem 32). The counterexample does not explain why the result is wrong and where the proof is not correct. In this paper, we answer these questions and moreover we give a corrected statement for the fractional Noether's theorem adapting the Frederico and Torres strategy of proof. Our discussion is made with respect to the fractional Noether's Theorem formulated in \cite{FT} but all our remarks and results applies also to \cite{FT2}.\\

As we will see, these questions lead to many difficulties which are not only interesting with respect to the fractional Noether theorem, but for all the generalizations proved by some authors using the same method, in particular in the discrete case (see \cite{BT} and \cite{ACP}). Precisely, Frederico and Torres generalize a method proposed by J.Jost and X. Li-Jost in \cite{jost} in the classical case. The idea is simple. The Noether's theorem is simple to prove in the case of transformations which do not depends on time. In order to cover the case of time dependent transformation, one introduces an extended Lagrangian taking the time as a new variable and then using the Noether's theorem in the autonomous case. The scheme of proof given in \cite{jost} is not very detailed and some points are omitted. These difficulties can be easily solved in the classical case and are related to standard results. However, trying to generalize this approach in the fractional case lead to serious difficulties. Forgetting for a moment the invariance condition and only concentrating on the proof given by Frederico and Torres, several points invalidate parts of the computations made in \cite{FT}. These difficulties arise in all the generalizations of the Jost's method. However, the fractional setting is probably the worth one in solving these problems.\\

The plan of the paper is as follows. In Section \ref{sec:intro}, first we give some preliminary information about fractional operators and then we remind the cases of fractional Noether's theorem for Lagrangian systems invariant under the action of one parameter group without time transformation.
In Section \ref{sec:inv} we remind the definition of fractional Lagrangian systems and the definition of invariance by a special class of symmetry group of transformations used. Already in this part, we discuss particular difficulties related with the definition of invariance used in \cite{FT}. Section \ref{sec:jost} is devoted to the method of J. and L. Jost to prove Noether's theorem. First we briefly describe the method in the classical case and explain the points which are not given in \cite{jost} and are sources of ambiguities. Finally, in Subsection \ref{sec:fractional}, we explain how the Jost's method can be generalized in order to cover the fractional case and in Subsection \ref{sec:noether} we state the fractional Noether's theorem that one obtain in this case. In Section \ref{sec:numerics}, we give some numerical simulations supporting our results.

\section{Reminder about fractional Lagrangian systems and invariance} \label{sec:intro}

 We denote by $\CC^k$, $k\in \N\cup\{\infty\}$, the class of regularity of functions and let $\CC^k([a,b],\R^n)$ denotes the set of all functions of class $\CC^k$ defined on $[a,b]$ with values in $\R^n$ and $a, b$ are two real numbers such that $a<b$.

\subsection{Preliminaries on fractional operators}
Before presenting the main idea of the paper, we introduce preliminary information about fractional operators. For a function $f:[a,b]\rightarrow \R$ we define :
\begin{definition}
\label{defintl}
The left (respectively right) Riemann-Liouville fractional integral operator of order $\alpha>0$ is defined by
\begin{equation} \label{rll}
I^{\alpha}_{a+}f(t)=\frac{1}{\Gamma(\alpha)}\int_{a}^t \frac{f(s)}{(t-s)^{1-\alpha}}ds,
\end{equation}
respectively
\begin{equation} \label{rll}
I^{\alpha}_{b-}f(t)=\frac{1}{\Gamma(\alpha)}\int_{t}^b \frac{f(s)}{(s-t)^{1-\alpha}}ds,
\end{equation}
for $t\in [a,b]$, where $\Gamma(\cdot)$ is the gamma function.
\end{definition}
The fractional derivative is defined by composing the above fractional integrals and the classical derivative of integer order :
\begin{definition} \label{fracderiv}
Let $t\in [a,b]$ and $\alpha\in(0,1]$ then we define
\begin{itemize}
\item the left and right Riemann-Liouville fractional derivative of order $\alpha$ :
\begin{equation} \label{rll}
D^{\alpha}_{a+} f(t)=\left( \frac{d}{dt} \circ I_{a+}^{1-\alpha}\right) f(t)=\frac{1}{\Gamma(1-\alpha)}\frac{d}{dt}\int_{a}^t \frac{f(s)}{(t-s)^{\alpha}}ds,
\end{equation}
\begin{equation} \label{rlr}
D^{\alpha}_{b-} f(t)=\left( -\frac{d}{dt} \circ I_{b-}^{1-\alpha}\right) f(t)=\frac{1}{\Gamma(1-\alpha)}\frac{d}{dt}\int_{t}^b \frac{f(s)}{(s-t)^{\alpha}}ds,
\end{equation}

\item the left and right Caputo fractional derivative of order $\alpha$ :
\begin{equation} \label{capl}
\di_c D^{\alpha}_{a+} f(t)=\left(I_{a+}^{1-\alpha} \circ \frac{d}{dt} \right) f(t)=\frac{1}{\Gamma(1-\alpha)}\int_{a}^t \frac{1}{(t-s)^{\alpha}}f'(s)ds,
\end{equation}
\begin{equation} \label{capr}
\di_c D^{\alpha}_{b-} f(t)=\left(I_{b-}^{1-\alpha} \circ \frac{d}{dt} \right) f(t)=\frac{1}{\Gamma(1-\alpha)}\int_{a}^t \frac{1}{(s-t)^{\alpha}}f'(s)ds.
\end{equation}
\end{itemize}
\end{definition}

Note that, for every $0<\alpha<1$ and $x\in AC([a,b],\R^n)$ the above derivatives are defined almost everywhere on the interval $[a,b]$. Moreover we have the following relations between Caputo and Riemann-Liouville definitions :
\begin{equation}
\label{relation}
\begin{split}
D_{a+}^{\alpha} x & =\di_c D^{\alpha}_{a+}+\frac{(t-a)^{-\alpha}}{\Gamma(1-\alpha)}x(a),\\
D_{b-}^{\alpha} x & =\di_c D^{\alpha}_{b-}+\frac{(b-t)^{-\alpha}}{\Gamma(1-\alpha)}x(b).
\end{split}
\end{equation}

\subsection{Fractional Lagrangian and Euler-Lagrange equations}

The function $L$
$$
\begin{array}{rlcl}
L: & [a,b]\times \R^n\times \R^n & \longrightarrow & \R\\
 & (t,x,v) & \longrightarrow & L(t,x,v)
\end{array}
$$
 is said to be a Lagrangian function if $L$ is of class $\CC^2$ with respect to all its arguments. The Lagrangian function $L$ defines a fractional Lagrangian $\mathcal{L}$ for $x\in \CC^1$

\begin{equation}
\label{lag}
\mathcal{L}_{\alpha,[a,b]}(x)=\int_a^b L\left(t,x(t),\di_{c}D_{a+}^{\alpha}x(t)\right)dt.
\end{equation}

Let us denote by $\CC^1_0 ([a,b])$ the set of all functions of class $\CC^1$ vanishing at the ends of the interval $[a,b]$. We define $E\in C^1([a,b],\R^n)$ as a nonempty subset open in the $\CC^1_0 ([a,b])$-direction.

\begin{theorem} \label{critical}
\cite{bo}
Let $0<\alpha<1$, then $x\in E$ is a critical point of $\mathcal{L}$ if and only if $x$ is a solution of the fractional Euler-Lagrange equation:
\begin{equation}
\label{fel}
D_{b-}^{\alpha}\Big(\frac{\partial L}{\partial v}(t,x(t),\di_c D_{a+}^{\alpha} x(t)\Big)+\frac{\partial L}{\partial x}(t,x(t),\di_{c}D_{a+}^{\alpha} x(t))=0,
\end{equation}
for every $t\in [a,b]$.
\end{theorem}

\subsection{The classical fractional Noether's theorem}
First, we remind the classical Noether’s theorem providing a conservation law for Lagrangian systems invariant under the action of one parameter group of diffeomorphisms with no transformation in time. \\

Precisely, let us consider the local group of transformations $\phi_s :\R^n \mapsto \R^n$, $s\in \R$ such that the functional $\mathcal{L}$ is invariant, i.e.
\begin{equation}
\label{Linv0}
\int_{t_a}^{t_b} L\left(t,x(t),\di_{c}D_{a+}^{\alpha}x(t)\right)dt=\int_{t_a}^{t_b} L\Big(t,\phi_s(x)(t),\di_c D_{a+}^{\alpha}(\phi_s(x))(t)\Big)dt,
\end{equation}
where $[t_a,t_b]\in[a,b]$. In this case we have the following well known result (see \cite{ata},\cite{FT},\cite{BCG}):

\begin{theorem}
\label{twCNT}
Let $\mathcal{L}$ be a fractional Lagrangian functional given by (\ref{lag}) invariant under the local transformation group $\{ \phi_s\}_{s\in \R}$, then the following equality holds for every solution of (\ref{fel}):
\begin{equation}
\di\frac{\partial L}{\partial v} (\star ) \cdot \di_c D_{a+}^{\alpha}
\left (
 \di\frac{d}{ds} \left (
\phi_s (x) \right )\mid_{s=0} \right )-D_{b-}^{\alpha}\left(\di\frac{\partial L}{\partial v} (\star )\right) \cdot \di\frac{d}{ds} \left (
\phi_s (x) \right )\mid_{s=0} =0,
\end{equation}
where $\left(\star\right) =\left ( t,x(t),\di_c D^{\alpha}_{a+}x(t)\right)$.
\end{theorem}

Even if there exists no Leibniz relation for fractional derivatives, one can deduce from the previous equality a first integral (see \cite{BCG}). With $\alpha=1$ Theorem \ref{twCNT} covers the classical Noether's theorem :
\begin{theorem}
\label{twCNT1}
Let $\mathcal{L}$ be a fractional Lagrangian functional given by (\ref{lag}) invariant under the local transformation group $\{ \phi_s\}_{s\in \R}$. Then for every solution of Euler-Lagrange equation :
\begin{equation} \label{CEL}
\frac{\partial L}{\partial x}(t,x(t),\dot{x}(t))=\frac{d}{dt}\left(\frac{\partial L}{\partial v}(t,x(t),\dot{x}(t)) \right),
\end{equation}
the following equality
\begin{equation}
\frac{d}{dt}\left(\frac{\partial L}{\partial v}(t,x(t),\dot{x}(t))\cdot \frac{d}{ds}\phi_s(x)|_{s=0}\right) =0
\end{equation}
holds, where $\dot x$ means the classical derivative of $x$.
\end{theorem}

\subsection{The Noether's theorem for Lagrangian mixing classical and fractional derivatives}

In order to generalize the Jost's method we need an extension of the previous results in the case where the Lagrangian depends both on the classical derivative and the fractional one. This extension is already done in \cite{FT}.\\

Let us consider a Lagrangian $L:[a,b] \times \R^n \times \R^n \rightarrow \R^n$, $L(t,x,w,v)$, and the fractional functional
\begin{equation}
\label{lag2}
\mathcal{L}_{\alpha,[a,b]}(x)=\int_a^b L\left(t,x(t),\dot{x} (t), \di_{c}D_{a+}^{\alpha}x(t)\right)dt .
\end{equation}
Then the critical points of $\mathcal{L}$ are given by the solution of the following mixed fractional Euler-Lagrange equation :
\begin{equation}
\label{fel2}
\di\frac{d}{dt} \Big(\frac{\partial L}{\partial w}(\star ) \Big )
=
D_{b-}^{\alpha}\Big(\frac{\partial L}{\partial v}(\star) \Big)
+
\frac{\partial L}{\partial x}(\star),
\end{equation}
where $\left(\star\right) =\left(t,x(t),\dot{x},\di_c  D_{a+}^{\alpha} x(t) \right)$.

We consider the local group of transformations $\phi_s :\R^n \mapsto \R^n$, $s\in \R$ such that the functional $\mathcal{L}$ given by (\ref{lag2}) is invariant, i.e.
\begin{equation}
\label{Linv0}
\int_{t_a}^{t_b} L\left(t,x(t),\dot{x}(t),\di_{c}D_{a+}^{\alpha}x(t)\right)dt=\int_{t_a}^{t_b} L\Big(t,\phi_s(x)(t),\frac{d}{dt}\phi_s(x)(t),\di_c D_{a+}^{\alpha}(\phi_s(x))(t)\Big)dt,
\end{equation}
where $[t_a,t_b]\in[a,b]$.

\begin{theorem}
\label{mix}
Let $\mathcal{L}$ defined by (\ref{lag2}) be a fractional Lagrangian functional invariant under the local transforation group $\{\phi_s\}_{s\in\R}$, then the following relation
\begin{equation}
\frac{\partial L}{\partial x}(\star)\cdot \frac{d\phi_s(x)}{ds}|_{s=0}+\frac{\partial L}{\partial v}(\star)\cdot \di_c D_{a+}^{\alpha}\left(\frac{d\phi_s(x)}{ds}|_{s=0}\right)+\frac{\partial L}{\partial w}(\star)\frac{d}{dt}\left( \frac{d\phi_s(x)}{ds}|_{s=0}\right)=0,
\end{equation}
where $(\star)=\left(t,x(t),\dot{x}(t),\di_{c}D_{a+}^{\alpha}x(t)\right)$, holds for every solution of the Euler-Lagrange equation (\ref{fel2}).
\end{theorem}

\section{Invariance of functionals and variational symmetries} \label{sec:inv}

\subsection{Variational symmetries}

We refer to the classical book of P.J.Olver \cite{olver} for more details in particular Chapter 4. In the following, we consider a special class of symmetry groups of differential equations called {\it projectable} or {\it fiber-preserving} (see \cite{olver},p.93) and given by
\begin{equation}
\label{group}
\begin{array}{rlcl}
\phi_s:& [a,b]\times \R^n & \longrightarrow & \R\times \R^n\\
 & (t,x) &  \longrightarrow & (\varphi_s^0(t),\varphi_s^1(x)),
\end{array}
\end{equation}
where $\{\phi_s\}_{s\in\R}$ is a one parameter group of diffeomorphisms satisfying $\phi_0=\mathds{1}$, where $\mathds{1}$ is the identity function. The associated {\it infinitesimal (or local) group action} (see \cite{olver},p.51) or transformations is obtained by making a Taylor expansion of $\phi_s$ around $s=0$:
\begin{equation} \label{phi}
\phi_s(t,x)=\phi_0(t,x)+s\di\frac{\partial \phi_s(t,x)}{\partial s} |_{s=0}+o(s).
\end{equation}
The {\it transform} (see \cite{olver},p.90) of a given function $x(t)$ identified with its graph $\Gamma_x =\{ (t,x(t)),\ t\in [a,b] \}$ by $\phi_s$ is easily obtained introducing a new variable $\tau$ defined by $\tau=\varphi_s^0(t)$. The transform of $x$ denoted by $\tilde{x}$ is then given by
$$
\tau\longrightarrow (\tau,\varphi_s^1\circ x\circ (\varphi_s^0)^{-1}(\tau)).
$$

\begin{remark}
In general, the transform of a given function is not so easy to determine explicitely (see \cite{olver}, Example 2.21, p.90-91) and one must use the implicit function theorem in order to recover the transform of $x$. This is precisely the reason why we restrict our attention to projectable or fiber-preserving symmetry groups. 
\end{remark}

We have the following fractional generalization of the definition of a {\it variational symmetry group} of a functional (see \cite{olver}, Definition 4.10 p.253):

\begin{definition}[Variational symmetries]
\label{def:1}
The local group of transformation $\phi_s$ is a variational symmetry group of the functional (\ref{lag}) if whenever $I=[t_a ,t_b ]$ is a subinterval of $[a,b]$ and $x$ is a smooth function defined over $I$ such that its transform under $\phi_s$ denoted by $\tilde{x}$ is defined over $\tilde{I}=[\mu_a ,\mu_b ]$ which is a subset of $\phi_s^0 ([a,b])=[\tau_{a},\tau_{b}]$, then
\begin{equation}
\label{Linv1}
\mathcal{L}_{\alpha,a,I}(x)=\mathcal{L}_{\alpha,\tau(a),\tilde{I}}(\tilde{x}).
\end{equation}
\end{definition}

It is interesting to give an explicit formulation of this definition. Indeed, according to definition (\ref{lag}) we can write (\ref{Linv1}) as
\begin{equation}
\label{Linv2}
\int_{t_a}^{t_b} L\left(t,x(t),\di_{c}D_{a+}^{\alpha}x(t)\right)dt=\int_{\mu_{a}}^{\mu_{b}} L\left(\tau,\varphi_s^1\circ x\circ (\varphi_s^0)^{-1}(\tau),\di_c D_{\tau_{a}+}^{\alpha}\left(\varphi_s^1\circ x\circ (\varphi_s^0)^{-1}(\tau)\right)\right)d\tau.
\end{equation}
The main point is that the explicit form of the integrand of the functional $\mathcal{L}_{\alpha,a,[t_a ,t_b ]}(x)$ depends on $a$ via the {\it base point} chosen for the fractional derivative. As a consequence, one {\it must} change the base point of the fractional derivative under the infinitesimal group action. This explain the change from the fractional derivative $\di_{c}D_{a+}^{\alpha}$ to $\di_c D_{\tau_{a}+}^{\alpha}$ in the previous expression.

\begin{remark}
The previous definition is in accordance with the one given by Atanackovic et al. in (\cite{ata}, Definition 10,p.1511).
\end{remark}

\begin{remark}
The fractional case with time transformation is very different from the classical case but also from the autonomous fractional case. Indeed, in the classical case, the integrand does not depend on the interval due to the local character of the classical derivative. Moreover, in the autonomous fractional case, the base point is not changed and as a consequence, if $x\in \mathcal{F}_{a,\alpha}$ is such that $\di_c D_{a+}^{\alpha} x$ is well defined then $\varphi_s^1\circ x\in \mathcal{F}_{a,\alpha}$ and the fractional derivative is also well defined with the same base point.
\end{remark}

\subsection{The Frederico-Torres definition of invariance}
\label{invFT}

In \cite{FT}, Frederico and Torres use a different definition. Indeed, in this case the authors do not change the base point for the fractional derivative in the definition they use for invariance of a functional (see Definition 16 p.840):

\begin{equation}
\label{LinvFT}
\int_{t_a}^{t_b} L\left(t,x(t),\di_{c}D_{a+}^{\alpha}x(t)\right)dt=\int_{\tau_{a}}^{\tau_{b}} L\left(\tau,\varphi_s^1\circ x\circ (\varphi_s^0)^{-1}(\tau),\di_c D_{a+}^{\alpha}\left(\varphi_s^1\circ x\circ (\varphi_s^0)^{-1}(\tau)\right)\right)d\tau.
\end{equation}

This means that their result is restricted to the case where $\phi_s^0 (a)=a$ for all $s\in \R$. This case was studied in (\cite{ata},Section 2.1 p.1507) and leads to a definition which is similar to \cite{FT} (see \cite{ata}, Definition 4,p.1509). 

\begin{lemma}[localization]
Let $\{ \phi_s^0 \}_{s\in \R}$ be a one parameter group of diffeomorphisms satisfying $\phi_s^0 (a)=a$ for all $s\in \R$, then we have 
\begin{equation}
\label{delfim1}
\phi_s^0 (t) =a + (t-a) \gamma_s (t) , 
\end{equation}
with $\gamma_s$ satisfying
\begin{equation}
\label{delfim2}
\gamma_{s+s'} (t) =\gamma_{s'} (t) \gamma_s ((t-a)\gamma_{s'} (t) +a) ,
\end{equation}
for all $s,s'\in \R$. 
\end{lemma}

\begin{proof}
The first part follows from the Hadamard Lemma and the second one from the group property.
\end{proof}

As we will see, these conditions implies strong constraints on the type of symmetries that one can consider.

\section{The Noether theorem and the Jost method in the fractional case} \label{sec:jost}

\subsection{A fractional Noether theorem}

Our aim in this Section is to give a new proof of the following fractional version of the Noether theorem:

\begin{theorem}[Fractional Noether theorem]
\label{main}
Suppose $G=\{ \phi_s (t,x)=(\phi_s^0 (t) ,\phi_s^1 (x) )\}_{s\in \R}$ is a one parameter group of symmetries of the variational problem $$\di\mathcal{L}_{\alpha,[a,b]}(x)=\di\int_a^b L\left(t,x(t),\di_{c}D_{a+}^{\alpha}x(t)\right)dt$$ 
such that 
\begin{equation}
\di\frac{d \phi_s^0}{dt} =K(s) ,
\end{equation}
where $K(s)$ is a function satisfying $K(0)=1$.
Let 
\begin{equation}
X= \zeta (t) \di\frac{\partial}{\partial t} +\xi (x) \di\frac{\partial}{\partial x} ,
\end{equation}
be the infinitesimal generator of $G$. Then, the function
\begin{equation}
\label{conslaw}
I(x)=L(\star ) \cdot\zeta +\di \int_a^t
\left [  
D_{b-}^{\alpha} \left [ \partial_v L (\star ) \right ] .\left ( \dot{x} \zeta -\xi\right  ) -\partial_v L (\star ). \left ( \zeta \cdot D_{a+}^{\alpha}[ \dot{x}] 
+\dot{\zeta} \cdot D_{a+}^{\alpha} [x] - D_{a+}^{\alpha} (\xi ) 
\right )
\right ] dt ,
\end{equation}
is a constant of motion on the solution of the fractional Euler-Lagrange equation (\ref{fel}).
\end{theorem}

One can of course directly prove this Theorem by differentiating the invariance relation with respect to $s$ and taking $s=0$ in the expression. Interested people will find such computations in the work of Atanakovic and al. \cite{ata}. Here, we follow a different strategy first proposed by Frederico and Torres in \cite{FT}. \\

Using the Euler-Lagrange equation (\ref{fel}), one can write the conservation law (\ref{conslaw}) as follows:

\begin{equation}
\label{conslaw2}
I(x)=L(\star ) \cdot\zeta +\di \int_a^t
\left [  
-\left [ \partial_x L (\star ) \right ] .\left ( \dot{x} \zeta -\xi\right  ) -\partial_v L (\star ). \left ( \zeta \cdot D_{a+}^{\alpha}[ \dot{x}] 
+\dot{\zeta} \cdot D_{a+}^{\alpha} [x] - D_{a+}^{\alpha} (\xi ) 
\right )
\right ] dt ,
\end{equation}
which is more useful from the computational point of view.\\

The condition concerning the symmetry group is of course restrictive but it covers already many interesting examples like the {\it translation in time} group given by $\phi_s^0 (t)=t+s$ or a more complicated one given by 
$\phi_s^0 (t)= t e^{-cs}=t-cts +o(s)$, where $c$ is a constant and used in (\cite{FT},Example 34,p.845). In the important case of the translation group in time, we obtain:

\begin{corollary}
\label{timespec}
Assume that the Lagrangian is independent of the time variable, then the quantity
\begin{equation}
I(x)=L(\star ) +\di \int_a^t
\left [  
D_{b-}^{\alpha} \left [ \partial_v L (\star ) \right ] .\dot{x} -\partial_v L (\star ). \cdot D_{a+}^{\alpha}[ \dot{x}]  
\right ] dt ,
\end{equation}
is a constant of motion on the solution of the fractional Euler-Lagrange equation (\ref{fel}).
\end{corollary}

We will use this result to test our theorem using numerical simulations.

\subsection{Reminder about the classical case}
\label{jostclassic}

In this Section, we consider the case $\alpha=1$. As recalled in the introduction, the basic idea behind the Jost method is to recover the Noether theorem for general transformations from the easier one corresponding to transformations without time. In the following, we indicate some steps in this method which lead to difficulties in the fractional case.\\

A first step is to reduce the invariance condition of the functional to an equality which can be understood as an invariance formula for transformations without transforming time, i.e. without changing the boundaries of integration. This is easily done using a change of variables. Indeed, posing
\begin{equation}
\tau=\varphi_s^0 (t ),
\end{equation}
in the right hand side of the invariance formula (\ref{Linv2}), one easily gets
\begin{equation}
\label{jost1}
\int_{a}^{b}L\left(t,x(t),\frac{dx(t)}{dt}\right)dt = \int_{a}^{b}L\left (\varphi^0_s(t),(\varphi^1_s\circ x)(t),\frac{d}{dt}\left(\varphi_s^1\circ x \right)(t) \frac{1}{\frac{d\varphi_s^0(t)}{dt}}\right )\frac{d\varphi_s^0(t)}{dt}dt .
\end{equation}
During the derivation of this equality, one uses a particular feature of the classical derivative which is the {\it chain rule} property, precisely we use the relation
\begin{equation}
\frac{d}{d\tau}\left(\varphi^1_s\circ x \circ(\varphi_s^0)^{-1}(\tau)\right)=\frac{d}{dt}\left(\varphi_s^1\circ x \right)(t) \frac{1}{\frac{d\varphi_s^0(t)}{dt}}.
\end{equation}
However, in the fractional calculus case this property of chain rule is known to be false and more difficult formula must be considered (we refer to \cite{aj2} for a general discussion about the algebraic relations that one can wait generalizing the notion of derivative to continuous functions).\\

Introducing the {\it extended Lagrangian} defined by
\begin{equation}
\label{extendjost}
\tilde{L} (\tau, (t,x),(w,v)) := L\left (
t,x,\di\frac{v}{w} \right ) \cdot w ,
\end{equation}
equation (\ref{jost1}) can be interpreted as the invariance of $\tilde{L}$ under the group of transformations without transforming time given by
\begin{equation}
\phi_s (t,x) =\left (
\varphi_s^0 (t), \varphi_s^1 (x)
\right ) ,
\end{equation}
over the set of solutions of the Euler-Lagrange equations associated to $\tilde{L}$ which satisfy the condition
\begin{equation}
t(\tau ):= \tau ,
\end{equation}
denoted by $U$ in the following. Indeed, over $U$ we have
\begin{equation}
\tilde{L} (\tau ,t(\tau ), x(\tau ) ,\dot{t} (\tau ) ,\dot{x}(\tau ) ) =
L(\tau ,x(\tau ) ,\dot{x} (\tau )) .
\end{equation}
As a consequence, we can rewrite equation (\ref{jost1}) as
\begin{equation}
\di\int_a^b
\tilde{L}(\tau ,t(\tau) ,x(\tau ) , \dot{t} (\tau) ,\dot{x} (\tau ) ) d\tau =
\di\int_a^b
L
\left (
\tau ,\phi_s (t(\tau) ,x(\tau )) ,\di\frac{d}{dt}
\left (
\phi_s (t(\tau) ,x(\tau ))
\right )
\right )
d\tau
.
\end{equation}
The proof of the Noether theorem then follows easily from the case of transformations without changing time, which ensures that the following quantity
\begin{equation}
I{(\tau,(t,x),(w,v))}=
\frac{\partial\tilde{L}}{\partial v}\left(t,x,w,v\right)\cdot \left.\frac{d\varphi^1_s(x)}{ds}\right|_{s=0}+
\frac{\partial\tilde{L}}{\partial w}\left(t,x,w,v\right)\cdot \left.\frac{d\varphi^0_s(x)}{ds}\right|_{s=0}
\end{equation}
is a first integral over $U$.

A simple computation leads to the classical form of the first integral for general transformations
\begin{equation}
I{(\tau,(t,x),(w,v))}=
\frac{\partial L}{\partial v}\left(t,x,v\right)\cdot \left.\frac{d\varphi^1_s(x)}{ds}\right|_{s=0}+\left(L\left(t,x,v\right)-v\frac{\partial L}{\partial v}\left(t,x,v\right)\right) \left.\frac{\varphi^0_s(x)}{ds}\right|_{s=0}.
\end{equation}

In order to be complete, one needs to check if the solutions $x(t)$ of the Euler-Lagrange equations associated to $L$ produce solutions of the form $(t(\tau)=\tau , x(\tau) )$ of the Euler-Lagrange equations associated to $\tilde{L}$. Indeed, this was implicitly assumed in the previous derivation. The Euler-Lagrange equations associated to $\tilde{L}$ are given by
\begin{equation}
\begin{split}
&\frac{d}{d\tau}\left[\frac{\partial \tilde{L}}{\partial v}(\tilde{\star}_\tau)\right]=\frac{\partial \tilde{L}}{\partial x}(\tilde{\star}_\tau), \\
&\frac{d}{d\tau}\left[\frac{\partial \tilde{L}}{\partial w}(\tilde{\star}_\tau)\right]=\frac{\partial \tilde{L}}{\partial t}(\tilde{\star}_\tau),
\end{split}
\end{equation}
where $(\tilde{\star}_\tau)=\left(t(\tau),x(t(\tau)),\frac{dt(\tau)}{d\tau}, \frac{dx(t(\tau))}{d\tau}\right)$.

A simple computation leads to
\begin{align}
&\frac{\partial \tilde{L}}{\partial t}(\tilde{\star}_\tau)= \frac{\partial L}{\partial t}(\star_\tau) \frac{dt(\tau)}{d\tau}, &\frac{\partial \tilde{L}}{\partial w}(\tilde{\star}_\tau) &= L\left(\star_\tau\right) - \frac{dx(t(\tau))}{d\tau}\frac{1}{\frac{d t(\tau)}{d\tau}} \frac{\partial L}{\partial v}(\star_\tau), \label{eq_partialtilde1}\\
&\frac{\partial \tilde{L}}{\partial x}(\tilde{\star}_\tau) = \frac{\partial L}{\partial x} (\star_\tau)\frac{dt(\tau)}{d\tau}, &\frac{\partial \tilde{L}}{\partial v} (\tilde{\star}_\tau)&= \frac{\partial L}{\partial v}(\star_\tau), \label{eq_partialtilde2}
\end{align}
where $(\star_\tau)=\left(t(\tau),x(t(\tau)),\frac{dx(t(\tau))}{d\tau}\frac{1}{\frac{d t(\tau)}{d\tau}}\right)$.

As a consequence, a path $(t(\tau )=\tau ,x(\tau ))$ is a solution of the Euler-Lagrange equations associated to $\tilde{L}$ if and only if
\begin{equation}
\label{EL1_final}
\frac{d}{d\tau}\left[\frac{\partial L}{\partial v}\left(\star_\tau\right)\right]=\frac{\partial L}{\partial x}\left(\star_\tau\right)
\end{equation}
and
\begin{equation}
\frac{d}{d\tau}L(\star_\tau)=\frac{\partial L}{\partial t}(\star_\tau)+\frac{d}{d\tau}\left(\frac{dx(\tau)}{d\tau} \frac{\partial L}{\partial v} (\star_\tau)\right),
\end{equation}
where $(\star_\tau)=\left( \tau,x(\tau ),\dot{x}(\tau) \right )$.\\

The first equation is exactly the Euler-Lagrange equation associated to $L$ for a path $x(\tau )$. Then, if we consider the restriction of $\tilde{\mathcal{L}}$ over $U$, this first equation is always satisfied.

For the second equation, we develop the left hand side which gives
\begin{equation}
\left(\frac{d}{d\tau}\left[\frac{\partial L}{\partial v}\left(\star_\tau\right)\right]-\frac{\partial L}{\partial x}\left(\star_\tau\right)\right)\frac{dx(\tau)}{d\tau}=0
\end{equation}
which is also always satisfied over $U$.

\begin{remark}
The main point is that this property comes from the specific expression of the total derivative of $L (\star_{\tau} )$. Here again, we need to use the chain rule. The same computation in the fractional case will lead some difficulties.
\end{remark}

\subsection{The fractional case} \label{sec:fractional}

In this Section, we extend the previous construction to the fractional case. We have divided the construction in several steps in order to discuss separately each of the difficulties involved.

\subsubsection{Step 1 - Construction of the extended Lagrangian}

As reminded in Section \ref{jostclassic}, the extended Lagrangian is obtained by rewriting the second term of equation (\ref{Linv2}) as an integral over the same interval $[t_a ,t_b ]$.\\

The problem is to be able to give an explicit expression for $\di_c D_{\tau_{a}+}^{\alpha}(\varphi_s^1\circ x\circ (\varphi_s^0)^{-1})(\tau)$ as an expression of $\di_c D_{\tau_{a}+}^{\alpha}(\varphi_s^1\circ x)(t)$.\\

Let $y=\varphi_s^1\circ x$, then
$$
\di_c D_{\tau_a+}^{\alpha}(y\circ (\varphi_s^0)^{-1})(\tau)=\frac{1}{\Gamma(1-\alpha)}\frac{d}{d\tau}\int_{\tau_a}^{\tau} \frac{1}{(\tau-p)^{\alpha}}(y\circ (\varphi_s^0)^{-1})(p)dp .
$$
We perform the change of variables $v=(\varphi_s^0)^{-1}(p)$ denoting $t=(\varphi_s^0)^{-1}(\tau)$. We then obtain:
\begin{equation}
\label{gf}
\di_c D_{\tau_a+}^{\alpha}(y\circ (\varphi_s^0)^{-1})(\tau)  = \frac{1}{\Gamma(1-\alpha)}\frac{d}{dt}\left(\int_a^t\frac{1}{(\varphi_s^0(t)-\varphi_s^0(v))^{\alpha}} y(v)\frac{d\varphi_s^0(v)}{dv}dv\right) \frac{1}{\frac{d\varphi_s^0(t)}{dt}} .
\end{equation}

We have here an illustration of the difficulties which come into play by adapting the Jost method. Indeed, without any assumptions, there exists no simple relations between the quantities $\di_c D_{\tau_a+}^{\alpha}(y\circ (\varphi_s^0)^{-1})(\tau)$ and $\di_c D^{\alpha}_{a+} (y)(t)$ contrary to the classical case. This is the classical chain rule problem with fractional derivatives. In order to solve this problem, we introduce a special class of symmetry groups.\\

First, we see that a relation between $\di_c D_{\tau_a+}^{\alpha}(y\circ (\varphi_s^0)^{-1})(\tau)$ and $\di_c D^{\alpha}_{a+} (y)(t)$ can be obtained if one consider one parameter group of diffeomorphisms $\{ \phi_s^0 \}$ satisfying 
\begin{equation}
\phi_s^0 (t) -\phi_s^0 (v) =\alpha (s) (t-v) , 
\end{equation}
for all $t,v\in [a,b]$ and all $s\in \R$. Specializing $v$ to a given value, we deduce that for all $s\in \R$, we have 
\begin{equation}
\phi_s^0 (t)=\alpha (s) t+\beta (s) ,
\end{equation}
i.e. that $\phi_s^0$ is an affine function for all $s\in \R$. Of course, the group property induces some constraints on the functions $\alpha$ and $\beta$. In particular, they must satisfy 
\begin{equation}
\label{rewriting}
\alpha (s+s')=\alpha (s)+\alpha (s') ,\ \ \mbox{\rm and}\ \ \ 
\beta (s+s') = \alpha (s) \beta (s') +\beta (s) , 
\end{equation}
with $\alpha (0)=1$ and $\beta (0) =0$. \\

The first condition of (\ref{rewriting}) implies that $\alpha$ must be an {\it exponential function}, i.e. that 
\begin{equation}
\alpha (s)=e^{\lambda s} , 
\end{equation}
for a certain $\lambda \in \R$. We then are leaded to the following class of symmetries groups:

\begin{definition}[Admissible groups]
A local group of transformations $\{ \phi_s =(\phi_s^0 ,\phi_s^1 )\}_{s\in \R}$ is said to be admissible,  if for all $s\in \R$, $\phi_s^0$ is an affine function of $t$ of the form 
\begin{equation}
\label{admigroup}
\phi_s^0 (t) = e^{\lambda s} t + \beta (s) , 
\end{equation}
with $\beta (s)$ satisfying $\beta (s+s') = e^{\lambda s} \beta (s') +\beta (s)$ for all $s,s'\in \R$ and $\beta (0)=0$.  
\end{definition}

Examples of admissible groups are given for example by the {\it translation group} $\varphi_s^0 (t) =t+s$ or a {\it scaling group} defined by $\phi_s^0 (t)=e^{cs} t$ where $c$ is a constant.\\

The main property of admissible groups is that a version of the chain rule property can be proved. Precisely, we have:

\begin{lemma}
Let $\{ \phi_s =(\phi_s^0 ,\phi_s^1 )\}_{s\in \R}$ be an admissible group. Then, we have for $0<\alpha \leq 1$ and for all $y\in AC ([a,b],\R^n)$:
\begin{equation}
\tag{$CR_{\alpha}$}
\label{cpalpha}
\di_c D_{\phi_s^0(a)+}^{\alpha}(y\circ (\phi_s^0)^{-1})(\tau) =
\di_c D^{\alpha}_{a+} (y)(t)\frac{1}{\left ( \di\frac{d\phi_s^0}{dt} \right )^{\alpha}} .
\end{equation}
\end{lemma}

\begin{remark}
The admissibility condition coupled with the localization assumptions (\ref{delfim1}) and (\ref{delfim2}) imply strong constraints. Precisely, we have:

\begin{lemma}
A one parameter group of diffeomorphisms $\{ \phi_s^0 \}$ acting on $[a,b]$ is admissible and satisfies the localization assumptions (\ref{delfim1}) and (\ref{delfim2}) if and only if it is of the form 
\begin{equation}
\mathbb{S}_{a,\lambda} =\{ \phi_s^0 (t) =e^{\lambda s}(t-a)+a \}_{s\in \R}
\end{equation}
for some $\lambda \in \R$. 
\end{lemma}
 
\begin{proof}
This is a simple computation.
\end{proof}

Many examples deal with the case $a=0$. In this case, the set of admissible groups satisfying the localization assumptions is reduced to the {\it group of dilatations}. 
\end{remark}

Under this assumption, one can easily rewrite the invariance condition as follows:

\begin{lemma}
\label{lemma}
Let $\{ \phi_s\}_{s\in \R}$ be an admissible local group of transformations.  If the Lagrangian functional $\mathcal{L}$ is invariant under the action of the one parameter group of diffeomorphisms $\{\phi_s\}_{s\in \R}$, then for any subinterval $I=[t_a ,t_b ]$ of $[a,b]$ and $x$ a smooth function defined over $I$ we have
\begin{equation}
\label{inv}
\int_{t_a}^{t_b} L\left(t,x(t),\di_{c}D_{a+}^{\alpha}x(t)\right)dt= \int_{t_a}^{t_b} L\left(\varphi_s^0(t),\varphi_s^1\circ x(t),\di_c D_{a+}^{\alpha}(\varphi_s^1\circ x)(t)\frac{1}{\left ( \frac{d\varphi_s^0}{dt}\right ) ^{\alpha} }\right)\frac{d\varphi_s^0(t)}{dt}dt.
\end{equation}
\end{lemma}

\begin{proof}
We perform the change of variable $t=\left ( \phi_s^0\right )^{-1} (\tau )$ in the integral
$$\int_{\mu_{a}}^{\mu_{b}} L\left(\tau,\varphi_s^1\circ x\circ (\varphi_s^0)^{-1}(\tau),\di_c D_{\tau_{a}+}^{\alpha}\left(\varphi_s^1\circ x\circ (\varphi_s^0)^{-1}(\tau)\right)\right)d\tau .$$
Using formula (\ref{gf}) and the assumption (\ref{cpalpha}), we deduce that
\begin{equation}
\label{Linv3}
\begin{split}
\int_{\mu_{a}}^{\mu_{b}} L\Big(\tau,\varphi_s^1\circ x\circ (\varphi_s^0)^{-1}(\tau),\di_c D_{\tau_{a}+}^{\alpha}(\varphi_s^1\circ x\circ (\varphi_s^0)^{-1}(\tau))\Big)d\tau\\
=\int_{t_a}^{t_b} L\left(\varphi_s^0(t),\varphi_s^1\circ x(t),\di_c D_{a+}^{\alpha}(\varphi_s^1\circ x)(t)\frac{1}{\left ( \frac{d\varphi_s^0(t)}{dt}\right )^{\alpha}}\right)\frac{d\varphi_s^0(t)}{dt}dt.
\end{split}
\end{equation}
The invariance condition (\ref{Linv2}) in Definition \ref{def:1} then reduces to (\ref{inv}).
\end{proof}

A useful consequence of the previous Lemma is the following classical but important result :

\begin{lemma}
Let $L$ be an autonomous Lagrangian, i.e. which does not depends on the time variable. Then, the associated functional is invariant under the {\it translation} group $\phi_s (t,x)=(t+s ,x)$.
\end{lemma}

This result is less evident when one is dealing with the initial definition.\\

The previous Lemma suggests to introduce the following {\it extended Lagrangian}:

\begin{definition}[Extended Lagrangian]
\label{extended}
Let $L (t,x,v)$ be a given admissible Lagrangian. The extended Lagrangian associated to $L$ and denoted by $\tilde{L} (\tau, (t,x),(w,v))$ is defined as follows
\begin{equation}
\label{defL}
\tilde{L}_{\alpha} (\tau, (t,x),(w,v)):=L\left ( t,x,\di\frac{v}{w^{\alpha}} \right ) \cdot w .
\end{equation}
\end{definition}

The Lagrangian functional associated to $\tilde{L}$ and denoted by $\tilde{\mathcal{L}}$ is given by
\begin{equation}
\label{Lxy}
\begin{split}
\tilde{\mathcal{L}}_{{\alpha},[a,b]}(t,x)&=\int_a^b \tilde{L}\left(t(\tau),x(t(\tau)),\frac{dt(\tau)}{d\tau},\di_c D^{\alpha}_{a+} x(t(\tau))\right)d\tau\\
&=\int_a^b \tilde{L}(t,x,w,v)d\tau.
\end{split}
\end{equation}

\begin{remark}
It must be noted that the Lagrangian (\ref{Lxy}) mixes the classical and fractional derivatives even if at the beginning the fractional Lagrangian problem was only dealing with fractional derivatives. As a consequence, we reduce the complexity from the non autonomous to autonomous Lagrangian but we increase the complexity from the functional point of view dealing with multiple sort of derivatives.
\end{remark}

The invariance of the functional $\mathcal{L}$ under the local symmetry group $\phi_s$ can then be rewritten as the invariance of $\tilde{\mathcal{L}}$ under an {\it autonomous} group action. Precisely, we have:

\begin{lemma}[Extended variational symmetries]
\label{extendsym}
Assume that the Lagrangian functional $\mathcal{L}$ associated to $L$ is invariant under an admissible local symmetry group $\{ \phi_s\}_{s\in\R}$. Then, the extended Lagrangian functional $\tilde{\mathcal{L}}$ associated to the extended Lagrangian $\tilde{L}$ satisfies
\begin{equation}
\label{inv1}
\tilde{\mathcal{L}}_{\alpha,[a,b]}(t,x)=\tilde{\mathcal{L}}_{\alpha,[a,b]}(\phi_s(t,x))=\tilde{\mathcal{L}}_{\alpha,[a,b]}(\varphi_s^0(t),\varphi_s^1(x)),
\end{equation}
over the set of paths $\tau \mapsto (t(\tau ), x(\tau ))$ satisfying $t(\tau )=\tau$ and $x(\tau )$ is a solution of the Euler-Lagrange equation associated to $L$. We denote by $U$ this set.
\end{lemma}

The restriction of the invariance relation on the set $U$ can not be avoid as in the classical case.

\subsubsection{Step 2 - Euler-Lagrange equations of the extended Lagrangian}

As already noted, the extended Lagrangian mixes the classical and fractional derivatives. Using formula (\ref{fel2}), we deduce that the Euler-Lagrange equations associated to the functional (\ref{Lxy})are given by:

\begin{equation}
\label{FEL}
\left\{\begin{split}
\frac{\partial \tilde{L}}{\partial x}(t,x,w,v)+D_{b-}^{\alpha}\left(\frac{\partial \tilde{L}}{\partial v}(t,x,w,v)\right) &=0 ,\\
\frac{\partial \tilde{L}}{\partial t}(t,x,w,v)-\frac{d}{d\tau}\left(\frac{\partial \tilde{L}}{\partial w}(t,x,w,v)\right) &=0.
\end{split}\right.
\end{equation}

The connection between the solutions of the initial fractional problem and those of the extended Lagrangian (\ref{defL}) are then :

\begin{lemma}[Euler-Lagrange equations for the extended Lagrangian]
The Euler-Lagrange equations associated to the extended Lagrangian of $L$ restricted to $U$ are given by
\begin{equation}
\label{FLL}
\begin{array}{r} (a)\\ \\(b)\end{array}\left\{\begin{split}
\frac{\partial L}{\partial x}(\star_{\tau})& +D_{b-}^{\alpha}\left(\frac{\partial L}{\partial v}(\star_{\tau})\right)=0\\
\frac{\partial L}{\partial t}(\star_{\tau})& -\frac{d}{d\tau}\left(L(\star_{\tau})-\alpha \di_c D^{\alpha}_{a+} x(\tau)\cdot\frac{\partial L}{\partial v}(\star_{\tau})\right)=0
\end{split}\right.
\end{equation}
where $(\star_{\tau})=(\tau,x(\tau),\di_c D^{\alpha}_{a+} x(\tau))$.
\end{lemma}

\begin{proof}
This follows from a simple computation. Indeed, we have
\begin{equation}
\label{extendedformula}
\begin{split}
\frac{\partial\tilde{L}}{\partial t}(t,x,w,v)&=\frac{\partial L}{\partial t}\left(t,x,\frac{v}{w^{\alpha}}\right)w,\;\;\;\; \frac{\partial\tilde L}{\partial x}(t,x,w,v)=\frac{\partial L}{\partial x}\left(t,x,\frac{v}{w^{\alpha}}\right)w,\\
\frac{\partial\tilde{L}}{\partial v}(t,x,w,v)&=w^{1-\alpha} \frac{\partial L}{\partial v}\left(t,x,\frac{v}{w^{\alpha}}\right),\;\;\;\; \frac{\partial \tilde{L}}{\partial w}(t,x,w,v)=L\Big(t,x,\frac{v}{w^{\alpha}}\Big)-\alpha\frac{v}{w^{\alpha}}\frac{\partial L}{\partial v}\left(t,x,\frac{v}{w^{\alpha}}\right)
\end{split}
\end{equation}
where  $(t,x,w,v)=\left(t(\tau),x(t(\tau)),\frac{dt(\tau)}{d\tau},\di_cD^{\alpha}_{a+} x(t(\tau))\right)$. This concludes the proof.
\end{proof}

We recognize the form already obtained in the classical case. The first equation of (\ref{FLL}) corresponds to the classical fractional Euler-Lagrange equation (\ref{fel}) associated to the Lagrangian $L$. The second equation (\ref{FLL}(b)) plays the same role as the energy in the classical case and is sometimes called {\it the second Euler-Lagrange equation} when $\alpha =1$. However, and this is the main difference independently of technical difficulties, this quantity is not {\it a priori} satisfied by solutions of the fractional Euler-Lagrange equation. As a consequence, the usual correspondence between the solution of extended Euler-Lagrange equations and the initial fractional Euler-Lagrange equation is not guaranteed. We will return on this condition in the following. We then have :

\begin{lemma}
\label{lemmace}
Solutions $x(t)$ of the fractional Euler-Lagrange equations (\ref{fel}) are solutions of the extended Euler-Lagrange equations (\ref{FLL}) if and only if they satisfy
\begin{equation}
\tag{$CE_{\alpha}$}
\label{lemext}
\frac{\partial L}{\partial t}(\star_{\tau}) -\frac{d}{d\tau}\left(L(\star_{\tau})-\di_cD^{\alpha}_{a+} x(\tau)\cdot\frac{\partial L}{\partial v}(\star_{\tau})\right)=0 ,
\end{equation}
where $(\star_{\tau})=(\tau,x(\tau),\di_cD^{\alpha}_{a+} x(\tau))$.
\end{lemma}

\begin{remark}
This is precisely this point which is not well developed in the derivation of the Noether's theorem in \cite{jost} and which is not discussed in the paper of Frederico and Torres \cite{FT}.
\end{remark}

However, condition (\ref{lemext}) is not a consequence of the fractional Euler-Lagrange equations for the initial Lagrangian $L$. We provide in the following a numerical example.

\subsection{About the second Euler-Lagrange equation in the fractional calculus of variations}

We consider the two-dimensional example of the quadratic Lagrangian on $[a,b]=[0,1]$ :
\begin{equation}
\label{LagHO}
\begin{array}{rlcl}
L:& [0,1]\times \R^2 \times \R^2 & \longrightarrow & \R\\
 & (t,x,v) &  \longrightarrow & \frac{1}{2}\left(\|x\|^2+\|v\|^2 \right),
\end{array}
\end{equation}
where $x=(x_1,x_2), v=(v_1,v_2)=(\di_c D^{\alpha}_{a+} x_1,\di_c D^{\alpha}_{a+} x_2)$. The corresponding Euler-Lagrange equation (\ref{fel}) for the Lagrangian $L$ defined above has the form
\begin{equation}
\label{ELH}
D_{b-}^{\alpha}\circ \di_c D^{\alpha}_{a+} x+x=0.
\end{equation}
The relation (\ref{lemext}) reduces to 
\begin{equation}
Q_{\alpha}(x):=\frac{1}{2}\left(\|x\|^2-\|v\|^2\right)=\mbox{\rm const} .
\end{equation}

In the classical case, with $\alpha=1$, the equation of motion has the form $\ddot{x}(t)=x(t)$ and the exact solution is given by $x(t)=(x_1(t),x_2(t))=(c_1 e^{t}+c_2 e^{-t},d_1 e^{t}+d_2 e^{-t})$ for $t\in [0,1]$, where $c_1, c_2, d_1, d_2\in \R$. Applying the classical Noether's Theorem \ref{twCNT} we obtain explicit constant of motion.\\

For the simulations, we consider the Dirichlet boundary conditions for the Euler-Lagrange equation (\ref{ELH}): $x(0)=(x_1(0),x_2(0))=(1,2)$, $x(1)=(x_1(1),x_2(1))=(2,1)$. In order to derive the approximate behavior of the fractional boundary problem, we discretize the integral form of the Euler-Lagrange equation (more information can be found in Appendix \ref{appendix:1}). Let $X=(X_1,X_2)$ denotes the approximation of the solution $x=(x_1,x_2)$. The behavior of the approximate solutions $X$ and the simulations for $Q_{\alpha}(X)$ with respect to different values of the order of derivative are presented on the Figures \ref{f1} and \ref{f2}, in the case of $\alpha=1$ and $\alpha\in (0,1]$, respectively.\\

In the classical case $\alpha=1$ we obtain the following:

\begin{figure}[!h]
\centerline{%
   \includegraphics[width=0.9\textwidth]{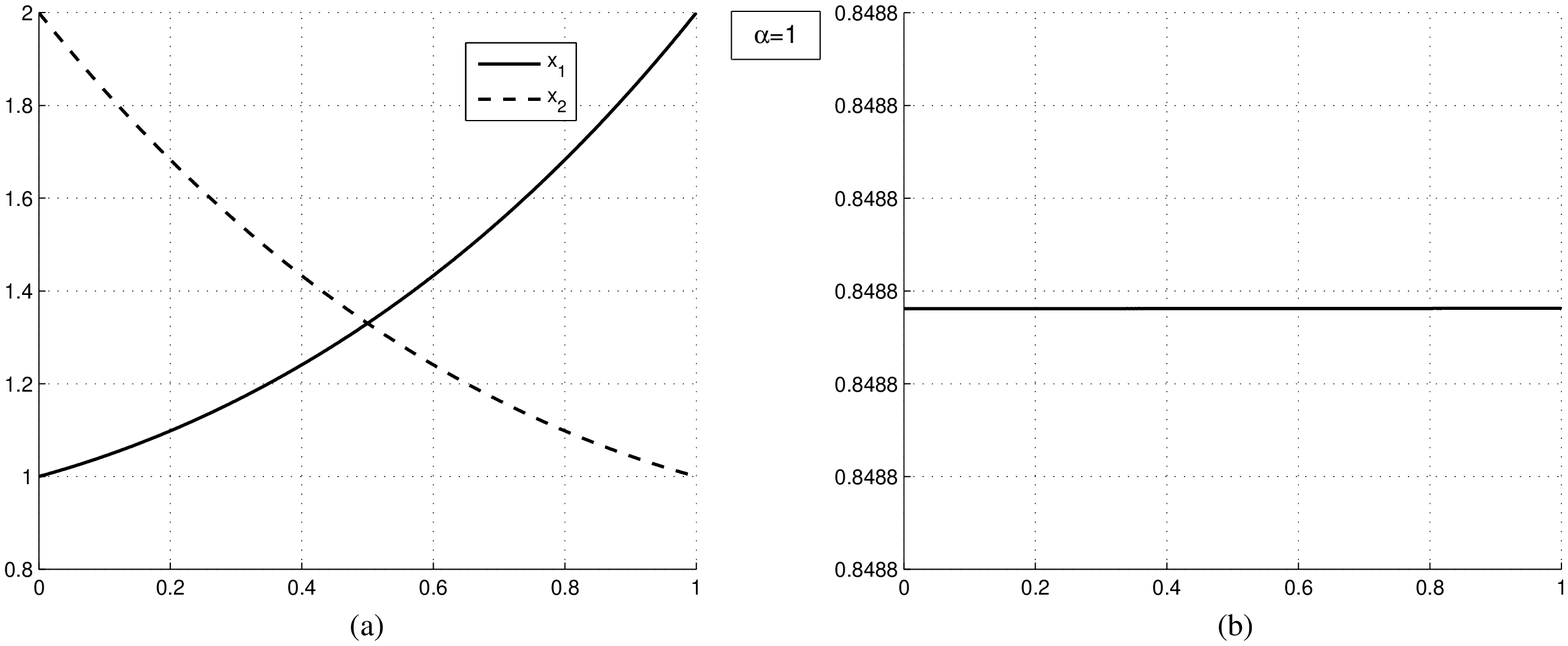}}
    \caption{Behavior of solution $X=(X_1,X_2)$ of (\ref{ELH}) with $\alpha=1$ is given on the sub-figure (a) and the constant of motion in this case can be observed on the sub-figure (b).}
    \label{f1}
\end{figure}

In the fractional case, the picture is very different:

\begin{figure}[!h]
\centerline{%
   \includegraphics[width=0.9\textwidth]{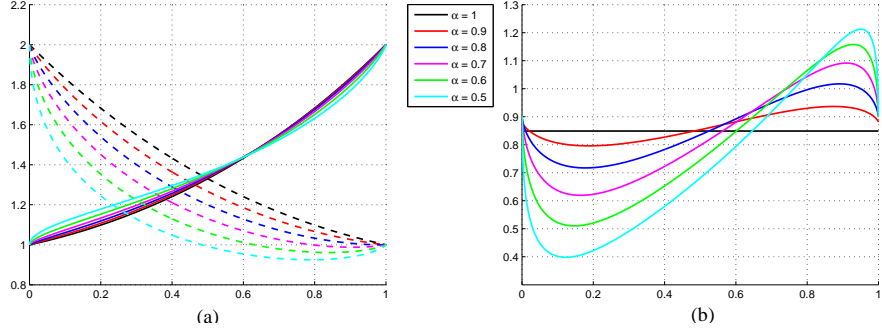}}
    \caption{Behaviors of solution $X=(X_1,X_2)$ of (\ref{ELH}) with respect to different values of $\alpha$ are given on the sub-figure (a) and respectively, behavior of $Q_\alpha(X)$ can be observed on the sub-figure (b).}
    \label{f2}
\end{figure}

The fractional version of the second Euler-Lagrange equation is clearly not satisfied at least in full generality. 

\subsection{Noether's theorem for the extended Lagrangian} 
\label{sec:noether}

In the case of extended Lagrangian the invariance condition (\ref{inv}) is a classical invariance relation for transformation groups which do not change the "time" variable. As remind in Section \ref{sec:intro}, this case was already studied by Frederico and Torres in \cite{FT} where they derive the corresponding Noether's theorem given in Theorem \ref{mix}. Using this result, we obtain:

\begin{theorem}[Noether's theorem for the extended Lagrangian] 
\label{NT}
Let $\tilde{L}$ is an invariant Lagrangian under the one-parameter group of diffeomorphisms $\{\Phi_s\}_{s\in\R}$. Then if $\tilde{\mathcal{L}}$ is a Lagrangian functional defined by Lagrangian $\tilde{L}$, then
\begin{equation} \label{thesis}
\begin{split}
 &\frac{d}{d\tau}\Big[ \frac{\partial \tilde{L}}{\partial w}\cdot\frac{d}{ds}(\varphi_s^0(t))|_{s=0}\Big]\\
 & + \bigg[\frac{\partial \tilde{L}}{\partial v}\cdot \di_cD_{a+}^{\alpha}\Big(\frac{d}{ds}(\varphi_s^1(x))|_{s=0}\Big)-D_{b-}^{\alpha}\left(\frac{\partial \tilde{L}}{\partial v}\right)\cdot\frac{d}{ds}(\varphi_s^1(x))|_{s=0}\bigg]=0
 \end{split}
 \end{equation}
over the solutions of the fractional Euler-Lagrange equations.
\end{theorem}

We are now ready to formulate the main result concerning the fractional generalization of the Jost's method.

\subsection{A first tentative : a weak fractional Noether's theorem}

We now derive the Noether's theorem which can be derived using a fractional version of the Jost's method :

\begin{theorem}[A fractional Noether's theorem]
Suppose $G=\{ \phi_s \}_{s\in \R}$ is a one parameter group of symmetries of the variational problem $\di\mathcal{L}_{\alpha,[a,b]}(x)=\di\int_a^b L\left(t,x(t),\di_{c}D_{a+}^{\alpha}x(t)\right)dt$ satisfying the chain rule property. Let 
\begin{equation}
X= \zeta (t) \di\frac{\partial}{\partial t} +\xi (x) \di\frac{\partial}{\partial x} ,
\end{equation}
be the infinitesimal generator of $G$. Assume also that for any solutions of the Euler-Lagrange equation we have 
\begin{equation}
\tag{$CE_{\alpha}$}
\label{lemext}
\frac{\partial L}{\partial t}(\star_{\tau}) -\frac{d}{d\tau}\left(L(\star_{\tau})-\di_cD^{\alpha}_{a+} x(\tau)\cdot\frac{\partial L}{\partial v}(\star_{\tau})\right)=0 ,
\end{equation}
where $(\star_{\tau})=(\tau,x(\tau),\di_cD^{\alpha}_{a+} x(\tau))$. Then we have:
\begin{equation}
\label{app1}
\frac{d}{dt}\left[\left(L-\di_cD_{a+}^{\alpha}x\cdot\frac{\partial L}{\partial v}\right)\zeta\right]+\left[\frac{\partial L}{\partial v}\cdot\di_cD_{a+}^{\alpha}\left(\xi\right)-D_{b-}^{\alpha}\left(\frac{\partial L}{\partial v}\right)\cdot\xi\right]=0.
\end{equation}
\end{theorem}
 
The proof follows from Theorem \ref{NT} and Lemma \ref{lemmace}.

\begin{remark}
In the case of $\alpha=1$, we recover the classical Noether theorem because condition ($CE_1$) and the chain rule property are automatically satisfied and the second term in (\ref{app1}) reduces to the total derivative of $\frac{\partial L}{\partial v}\cdot \xi$.
\end{remark}

In the case of $\alpha\ne 1$, with no transformation in time, the one parameter group satisfies the chain rule property as $\phi_s^0 (t)=t$ and $\zeta=0$. However, we {\bf do not recover the classical fractional Noether theorem}. Indeed, there is no reasons that the solutions of the Euler-Lagrange equations satisfy the condition (\ref{lemext}) and in fact, most of the time, they do not.\\

Moreover, in the case of $\alpha\ne 1$, with transformation in time, we consider as an example the special case of the translation group
$$
\varphi_s(t,x)=(t+s,x),
$$
from which we conclude that $\zeta=1$ and $\xi=0$. This group satisfies the chain rule property. Assuming that the condition (\ref{FEL}(b)) given by 
\begin{equation} 
\label{rel}
\frac{d}{dt}\left(L-\di_cD^{\alpha}_{a+} x\cdot\frac{\partial L}{\partial v}\right)=0,
\end{equation}
is satisfied, we derive as a conservation law the quantity (\ref{rel})!\\

These two remarks tell us that something is going wrong in the fractional generalization of the Jost's method. This point is discussed and solved in the next Section.

\subsection{A Jost's type proof of the fractional Noether theorem}

The previous tentative does not give the right answer. Where do we have made a too strong assumption in our computation ? As all the problems are clearly coming from the condition (\ref{lemext}) we must look at this condition and the reasons why we have introduced it. As we have said, the basic idea behind the Jost's method is to use the autonomous version of the fractional Noether theorem. In this case, one needs to ensure that the solutions that we consider are solution of the underlying Euler-Lagrange equations attached to the extended Lagrangian. However, doing so, we clearly ask for a too strong condition. The invariance relation by itself already provide a conserved quantity over the solution of the initial fractional Euler-Lagrange equation which is provided by the following {\it infinitesimal invariance criterion} (see \cite{FT},Theorem 17 p.840):

\begin{lemma}[Infinitesimal invariance criterion] 
If the Lagrangian function $\tilde{\mathcal{L}}$ is invariant under the one parameter group $\{ \phi_s =(\phi_s^0 ,\phi_s^1 ) \}_{s\in \R}$ then
we have 
\begin{equation}
\label{inficriterion}
\partial_t {\tilde{L}} . \di\frac{d\phi_s^0}{ds} \mid_{s=0} +
\partial_x {\tilde{L}} .\di\frac{d\phi_s^1}{ds} \mid_{s=0}
+
\partial_w {\tilde{L}} . \di\frac{d}{dt} \left ( \frac{d\phi_s^0}{ds} \mid_{s=0} \right )  +
\partial_v {\tilde{L}} .\di D_{a+}^{\alpha} \left ( \frac{d\phi_s^1}{ds} \mid_{s=0} \right )
=0
\end{equation}
\end{lemma}

As a consequence, using the extended variational symmetries Lemma \ref{extendsym} and formula (\ref{inficriterion}), we obtain:

\begin{lemma}
Suppose $G=\{ \phi_s \}_{s\in \R}$ is a one parameter group of symmetries of the variational problem $\di\mathcal{L}_{\alpha,[a,b]}(x)=\di\int_a^b L\left(t,x(t),\di_{c}D_{a+}^{\alpha}x(t)\right)dt$ satisfying the chain rule property. Let 
\begin{equation}
X= \zeta (t) \di\frac{\partial}{\partial t} +\xi (x) \di\frac{\partial}{\partial x} ,
\end{equation}
be the infinitesimal generator of $G$. Then, we have:
\begin{equation}
\label{funda}
\partial_t L .\zeta +\partial_x .\xi +L.\dot{\zeta} +
\partial_v L. 
\left ( 
-D_{a+}^{\alpha} x . \dot{\zeta} + D_{a+}^{\alpha} (\xi ) 
\right )
=0.
\end{equation}
\end{lemma}

The proof follows from simple computations using formula (\ref{extendedformula}).\\

The proof of the fractional Noether theorem now follows easily. Using the fact that 
\begin{equation}
\di\frac{d}{dt} \left ( L(\star ) \right ) =
\partial_t L (\star )+\partial_x L (\star ).\dot{x} +\partial_v L (\star ). D_{a+}^{\alpha}[ \dot{x}] ,
\end{equation}
we rewrite equation (\ref{funda}) as
\begin{equation}
\di\frac{d}{dt} \left ( L(\star ) \right ) \zeta -\partial_x L (\star ).\dot{x} \zeta -\partial_v L (\star ). D_{a+}^{\alpha}[ \dot{x}] \zeta
+\partial_x .\xi +L.\dot{\zeta} +
\partial_v L. 
\left ( 
-D_{a+}^{\alpha} x . \dot{\zeta} + D_{a+}^{\alpha} (\xi ) 
\right )
=0.
\end{equation}
Using the equality $\di\frac{d}{dt} \left ( L(\star ) \zeta \right ) = 
\di\frac{d}{dt} \left ( L(\star ) \right ) \zeta -L \dot{\zeta}$ and the fact that $x$ is a solution of the fractional Euler-Lagrange equation, we 
deduce that
\begin{equation}
\di\frac{d}{dt} \left ( L(\star ) \zeta \right ) + D_{b-}^{\alpha} \left [ \partial_v L (\star ) \right ] .\left ( \dot{x} \zeta -\xi\right  ) -\partial_v L (\star ). \left ( \zeta \cdot D_{a+}^{\alpha}[ \dot{x}] +
+\dot{\zeta} \cdot D_{a+}^{\alpha} [x] - D_{a+}^{\alpha} (\xi ) 
\right )
=0.
\end{equation}
A conservation law is then obtain integrating the previous expression between $a$ and $t$. We then obtain the function
\begin{equation}
I(x)=L(\star ) \cdot\zeta +\di \int_a^t
\left [  
D_{b-}^{\alpha} \left [ \partial_v L (\star ) \right ] .\left ( \dot{x} \zeta -\xi\right  ) -\partial_v L (\star ). \left ( \zeta \cdot D_{a+}^{\alpha}[ \dot{x}] +
+\dot{\zeta} \cdot D_{a+}^{\alpha} [x] - D_{a+}^{\alpha} (\xi ) 
\right )
\right ] dt .
\end{equation}
This concludes the proof of the fractional Noether theorem.

\section{Examples and numerical simulations}

\subsection{The fractional harmonic oscillator} 
\label{sec:numerics}

Let us consider the {\it fractional oscillator} studied in  (\cite{ata},Example $18$ page $1513$) for which the Lagrangian is given by:
\begin{equation}
L=\frac{1}{2}\left(\di_0D^{\alpha}_{t}u\right)^2-\omega^2 \frac{1}{2}u^2,
\end{equation}
where $\omega$ is a frequency. Initial conditions $u(0)= 0$ and $u'(0)= 1$. The Euler-Lagrange equation for such an $L$ is:
\begin{equation} \label{eq:1}
\di_tD^{\alpha}_1\left(\di_0D^{\alpha}_t u\right)=\omega^2u. 
\end{equation}
The fractional conservation law for (\ref{eq:1}):
\begin{equation} 
\label{quan}
\frac{1}{2}\left(\di_0D^{\alpha}_{t}u\right)^2-\omega^2 \frac{1}{2}u^2+\int_{0}^{t}\left(-\di_0D^{\alpha}_s u'\cdot \di_0D^{\alpha}_s u+u'\cdot \di_sD^{\alpha}_1\left(\di_0D^{\alpha}_s u\right)\right)ds=\mathrm{const}.
\end{equation}

We can check this result using numerical simulations. 

\begin{figure}[!h]
\centerline{%
   \includegraphics[width=1\textwidth]{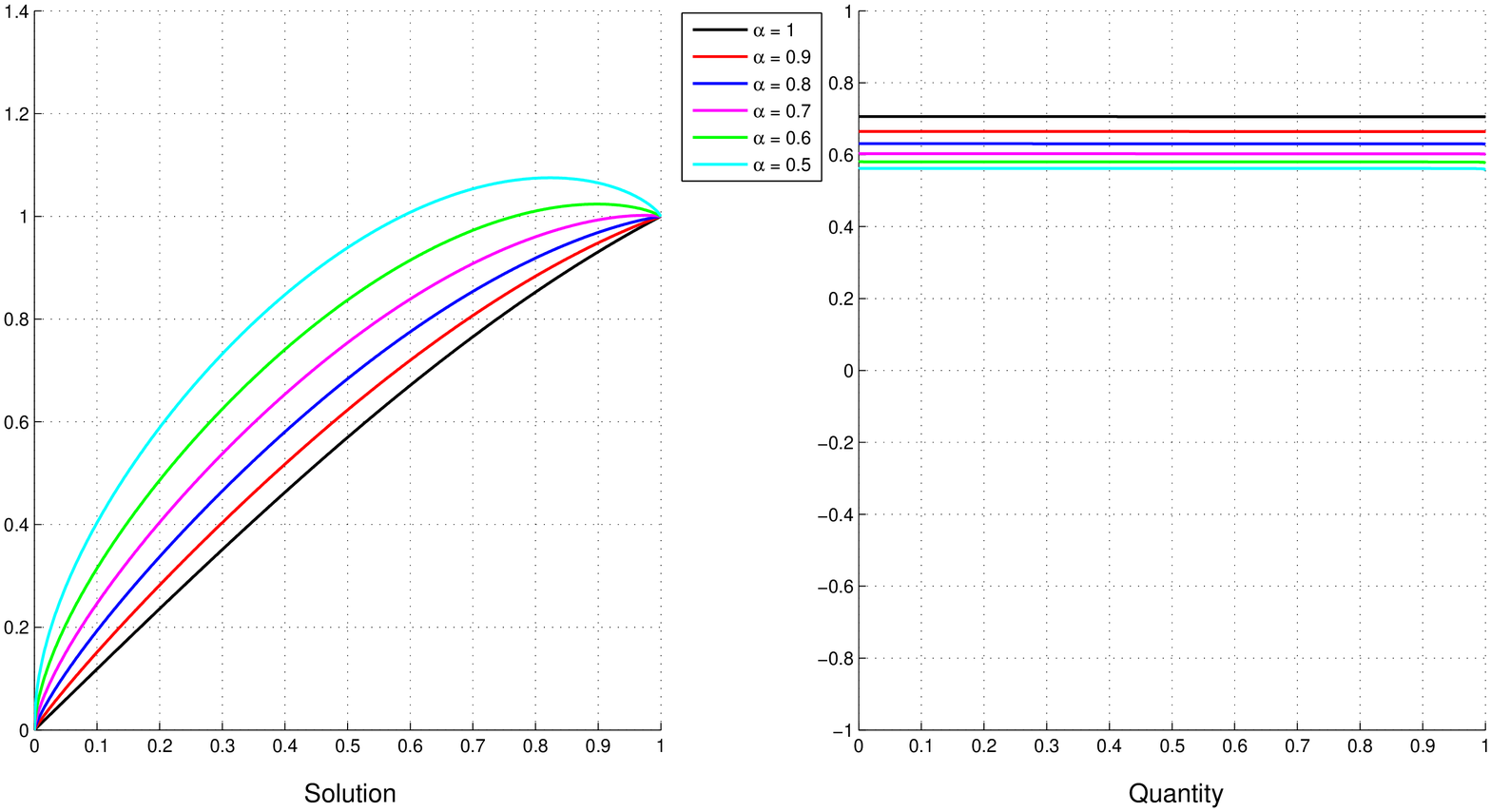}}
    \caption{Solution $u$ of the Euler-Lagrange equation (\ref{eq:1}) (left) and equivalent quantity (\ref{quan}) (right) with $\omega=1$. The range of quantity values is fixed: $[-1,1]$}
    \label{f1}
\end{figure}
\begin{figure}[!h]
\centerline{%
   \includegraphics[width=1\textwidth]{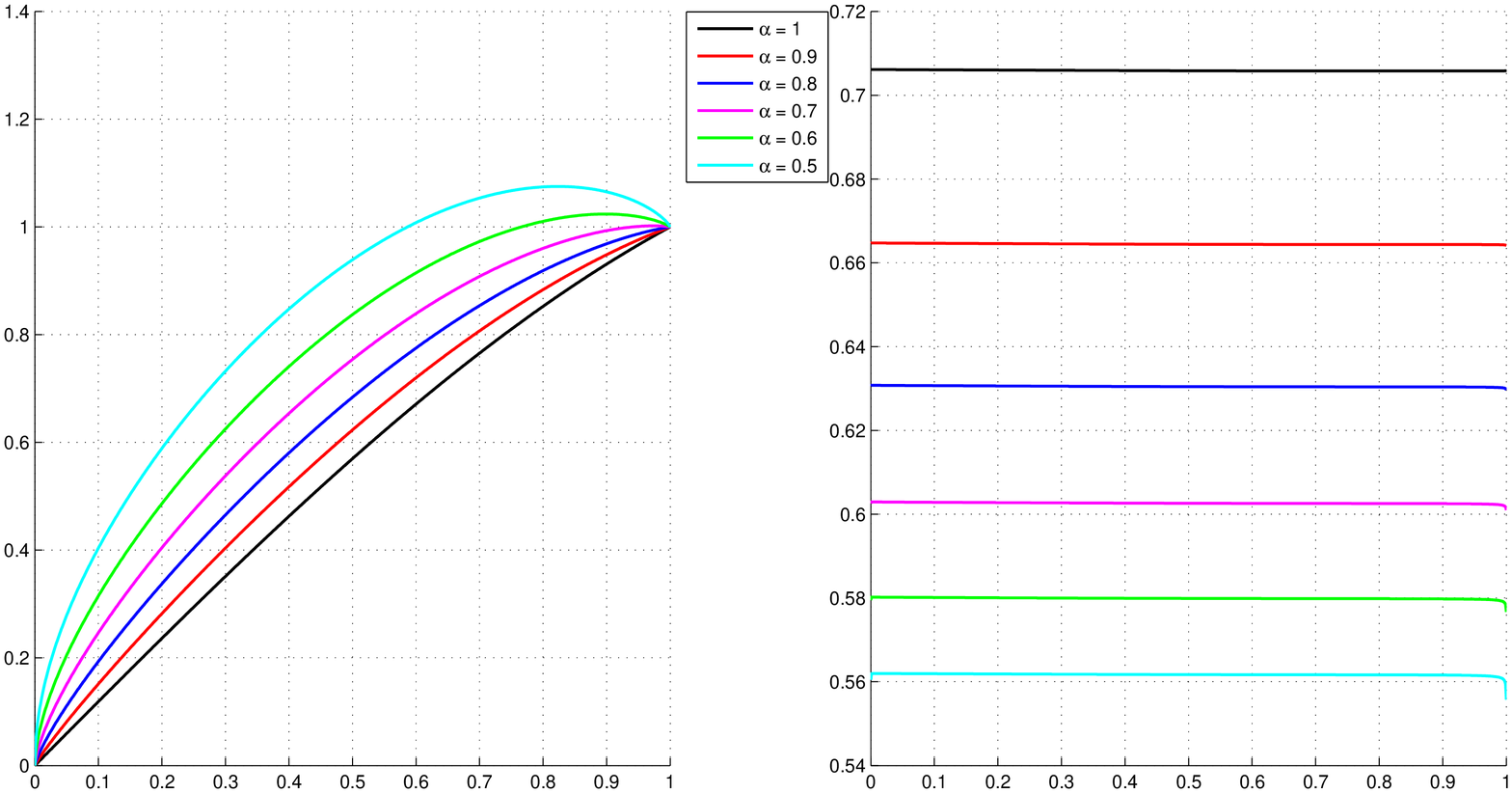}}
    \caption{Solution $u$ of the Euler-Lagrange equation (\ref{eq:1}) (left) and equivalent quantity (\ref{quan}) (right) with $\omega=1$. The range of quantity values is not fixed.}
    \label{f1}
\end{figure}
\begin{figure}[!h]
\centerline{%
   \includegraphics[width=1\textwidth]{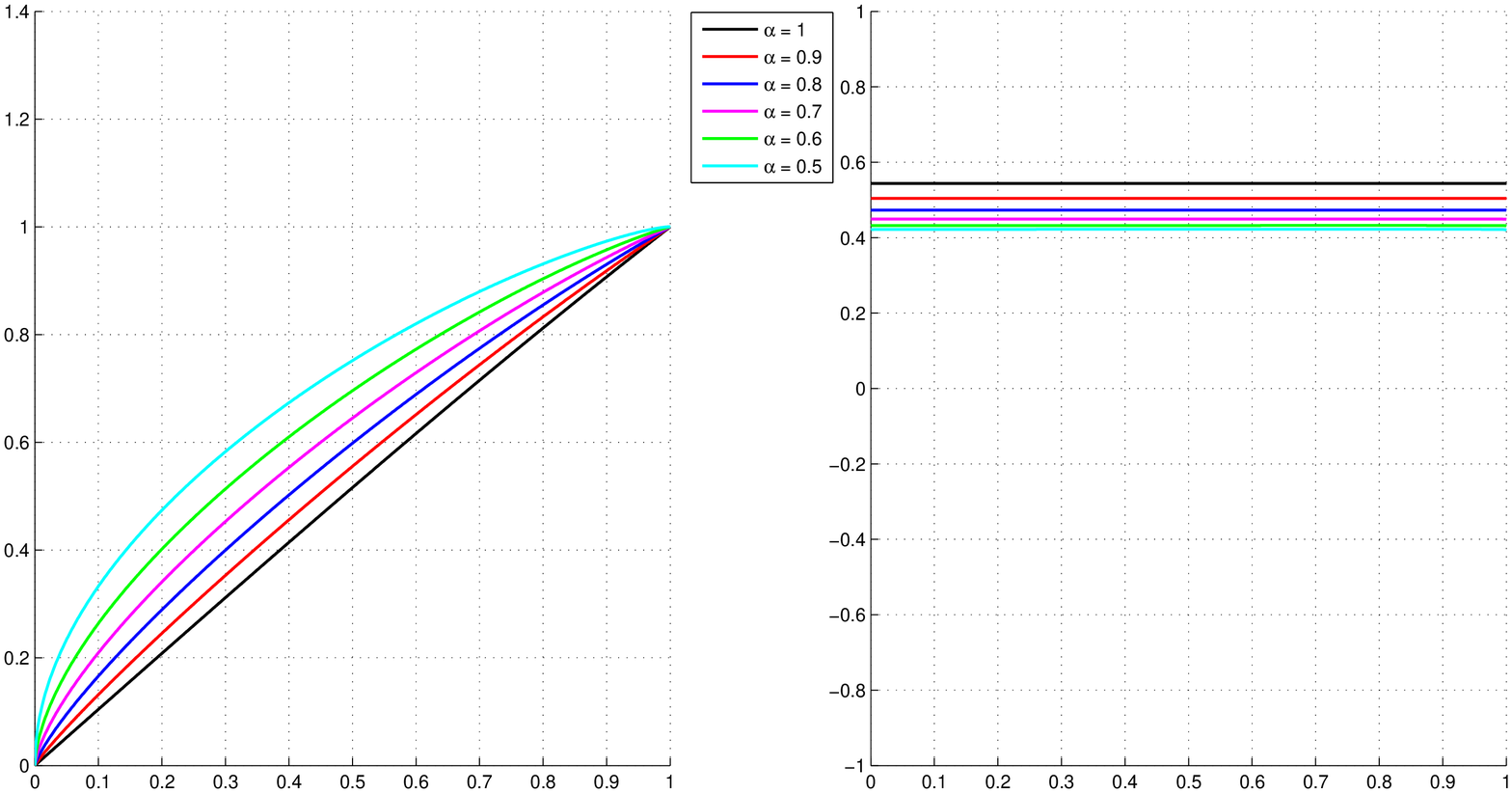}}
    \caption{Solution $u$ of the Euler-Lagrange equation (\ref{eq:1}) (left) and equivalent quantity (\ref{quan}) (right) with $\omega=0.5$. The range of quantity values is fixed: $[-1,1]$}
    \label{f1}
\end{figure}
\begin{figure}[!h]
\centerline{%
   \includegraphics[width=1\textwidth]{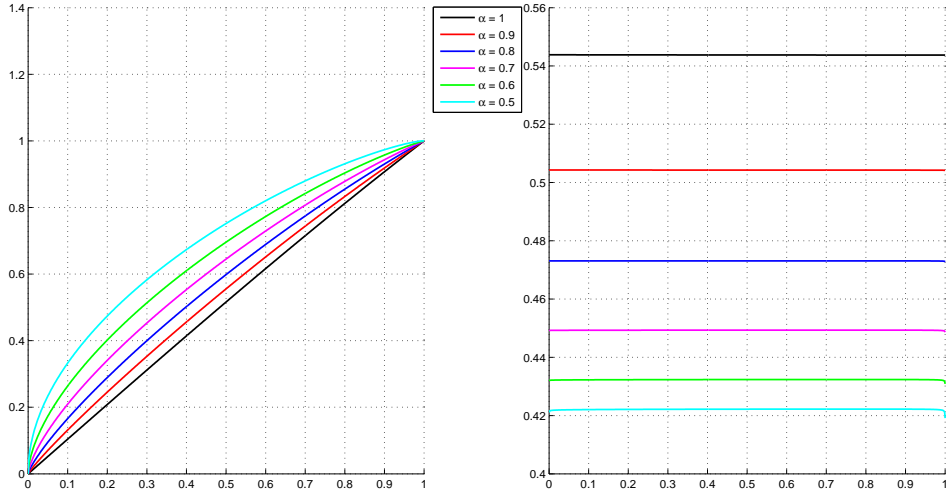}}
    \caption{Solution $u$ of the Euler-Lagrange equation (\ref{eq:1}) (left) and equivalent quantity (\ref{quan}) (right) with $\omega=0.5$. The range of quantity values is not fixed.}
    \label{f1}
\end{figure}

\subsection{Example 2}

We consider the one parameter $\alpha\in ]0,1]$ family of Lagrangian
\begin{equation}
\label{exemple2}
L_{\alpha} (x_1 ,x_2 ,v_1 ,v_2 )= v_1^{1/\alpha} x_2 -v_2^{1/\alpha} x_1  .
\end{equation}
As the Lagrangian is independent of the time variable, we can use Corollary \ref{timespec} to obtain a first conservation law. The quantity 
\begin{equation}
\left .
\begin{array}{lll}
I(q_1 ,q_2) & = &  
-\left [ q_2 \left ( D_{a+}^{\alpha} q_1 \right ) ^{1/\alpha} 
-q_1 \left ( D_{a+}^{\alpha} q_2 \right ) ^{1/\alpha} 
\right ] \\
& & +\di\int_a^t
\left ( 
- \dot{q_1} \left ( D_{a+}^{\alpha} q_2 \right )^{1/\alpha}   
+ \dot{q_2} \left ( D_{a+}^{\alpha} q_1 \right )^{1/\alpha}   
\right )
\, dt 
\\
& & 
+(1/\alpha)\di\int_a^t 
\left ( 
-q_2 \left ( D_{a+}^{\alpha} q_1 \right )^{(1-\alpha)/\alpha} D_{a+}^{\alpha} (\dot{q_1} ) 
+
q_1 \left ( D_{a+}^{\alpha} q_2 \right )^{(1-\alpha)/\alpha} D_{a+}^{\alpha} (\dot{q_2} ) 
\right )
\, dt
\end{array}
\right .
\end{equation}
is a conservation law.\\

The functional is also invariant under a more complicated symmetry groups.

\begin{lemma}
The fractional functional associated to $L_{\alpha}$ is invariant under the local group of transformations given by $\phi_s^0 (t)=t\, e^{-cs}$ and $\phi_s^1 (x)=x$ for $s\in \R$. 
\end{lemma}

\begin{proof}
We apply the invariance criterion given by Lemma \ref{lemma}. We have 
\begin{equation}
\left .
\begin{array}{lll}
L\left ( x(t) ,, \di\frac{1}{(e^{-cs})^{\alpha}} D_{a+}^{\alpha} (x) 
) \right )  \cdot 
e^{-cs} &  = & L(x, D_{a+}^{\alpha} (x) ) \cdot \di\frac{e^{-cs}}{(e^{-cs})^{\alpha/\alpha}} \\
& = & L(x,D_{a+}^{\alpha} (x)) ,
\end{array}
\right .
\end{equation}
which concludes the proof.
\end{proof}

We then can use the Theorem \ref{main} to obtain the following conservation law :

\begin{equation}
\left .
\begin{array}{lll}
I(q_1 ,q_2) & = &  
-\left [ q_2 \left ( D_{a+}^{\alpha} q_1 \right ) ^{1/\alpha} 
-q_1 \left ( D_{a+}^{\alpha} q_2 \right ) ^{1/\alpha} 
\right ] ct\\
& & -\di\int_a^t
\left ( 
- \dot{q_1} \left ( D_{a+}^{\alpha} q_2 \right )^{1/\alpha}   
+ \dot{q_2} \left ( D_{a+}^{\alpha} q_1 \right )^{1/\alpha}   
\right ) ct\, dt
\\
& & 
-\di\frac{1}{\alpha} \di\int_a^t 
\left ( 
-q_2 \left ( D_{a+}^{\alpha} q_1 \right )^{(1-\alpha)/\alpha} D_{a+}^{\alpha} (q_1 ) 
+
q_1 \left ( D_{a+}^{\alpha} q_2 \right )^{(1-\alpha)/\alpha} D_{a+}^{\alpha} (\dot{q_2} ) 
\right )
ct\, dt \\
& & 
-\di\frac{c}{\alpha} \di\int_a^t 
\left ( 
-q_2 \left ( D_{a+}^{\alpha} q_1 \right )^{(1-\alpha)/\alpha} D_{a+}^{\alpha} (q_1 ) 
+
q_1 \left ( D_{a+}^{\alpha} q_2 \right )^{(1-\alpha)/\alpha} D_{a+}^{\alpha} (q_2 ) 
\right )
\, dt \\
&
\end{array}
\right .
\end{equation}

\begin{appendix}
\section{Note on numerical solving of the Euler-Lagrange equation} \label{appendix:1}
In order to obtain approximate solution for the Euler-Lagrange equation (\ref{ELH}) we convert this equation into the integral form. First let us formulate useful composition rules between fractional operators according to Definitions \ref{defintl} and \ref{fracderiv} (see \cite{bo}):
\begin{lemma} \label{cr}
Let $\alpha\in (0,1)$ and $x\in AC([a,b],\R^n)$, then the following relations
\begin{itemize}
\item $I_{a+}^{\alpha}\circ \di_c D^{\alpha}_{a+}x=x-x(a)$,
\item $I_{a+}^{\alpha}\circ D^{\alpha}_{a+}x=x$
\end{itemize}
are satisfied almost everywhere.
\end{lemma}
In the case of the right operators the counterparts of this rules are also valid.
According to the definitions of fractional integrals (\ref{rll}) and (\ref{rlr}) we can conclude that for every constant $C\in \R$ we have
\begin{equation} \label{const}
I_{a+}^{\alpha}C=\frac{(t-a)^{\alpha}}{\Gamma(1+\alpha)}C,\qquad I_{b-}^{\alpha}C=\frac{(b-t)^{\alpha}}{\Gamma(1+\alpha)}C.
\end{equation}

Now, the integral form of the Euler-Lagrange equation (\ref{FEL}) :
\begin{equation}
\label{integral}
\begin{split}
x(t)+&I_{a+}^{\alpha}\circ I_{b-}^{\alpha} x(t)-\left(\frac{t-a}{b}\right)^{\alpha}\left[I_{a+}^{\alpha}\circ I_{b-}^{\alpha} x(t)\right]_{t=b}\\
=&\left(1-\left(\frac{t-a}{b}\right)^{\alpha}\right)x(a)+\left(\frac{t-a}{b}\right)^{\alpha}x(b)
\end{split}
\end{equation}
can be easily derived based on (\ref{const}), the relations between derivatives (\ref{relation}) and the composition rules defined in Lemma \ref{cr}. Note that, if we put $\alpha=1$ in (\ref{integral}), we obtain the integral form of the equation $\ddot{x}=x$.

For the purpose of discretization of the integral equation (\ref{integral}) we define the equidistant partition on $[a,b]$ : $h=(b-a)/N$, $t_k=a+kh$, for $k=0,\ldots,N$, $N\in \N$. On the subinterval $[t_i,t_{i+1}]$ we substitute the function $f$ by the arithmetic average of values $f(t_{i})$ and $f(t_{i+1})$. We derive the approximations of the integrals:
\begin{equation}
\begin{split}
I_{a+}^{\alpha}f(t)|_{t=t_k}=& \frac{1}{\Gamma(\alpha)}\sum_{i=0}^{k-1} \int_{t_i}^{t_{i+1}}\frac{f(s)}{(t_k-s)^{1-\alpha}}ds\\
\approx& \frac{1}{\Gamma(\alpha)}\sum_{i=0}^{k-1}\int_{t_i}^{t_{i+1}}\frac{1}{(t_k-s)^{1-\alpha}}\left( \frac{f(t_i)+f(t_{i+1})}{2} \right)ds=:\di^h I_{a+}^{\alpha}f(t_k)
\end{split}
\end{equation}
for $k=1,\ldots,N$, and
\begin{equation}
\begin{split}
I_{b-}^{\alpha}f(t)|_{t=t_k}=& \frac{1}{\Gamma(\alpha)}\sum_{i=k}^{N-1} \int_{t_i}^{t_{i+1}}\frac{f(s)}{(s-t_k)^{1-\alpha}}ds\\
\approx& \frac{1}{\Gamma(\alpha)}\sum_{i=k}^{N-1}\int_{t_i}^{t_{i+1}}\frac{1}{(s-t_k)^{1-\alpha}}\left( \frac{f(t_i)+f(t_{i+1})}{2} \right)ds=:\di^h I_{b-}^{\alpha}f(t_k)
\end{split}
\end{equation}
for $k=0,\ldots,N-1$, where the sub-integrals can be directly calculated. Then we obtain the following algebraic system of equations
\begin{equation}
\label{disc}
\begin{split}
X_0=&x(a),\\
X_k+&\di^h I_{a+}^{\alpha}\circ \di^h I_{b-}^{\alpha} X_k-\left(\frac{t_k-a}{b}\right)^{\alpha}\di^h I_{a+}^{\alpha}\circ \di^h I_{b-}^{\alpha} X_N\\
=&\left(1-\left(\frac{t_k-a}{b}\right)^{\alpha}\right)X_0+\left(\frac{t_k-a}{b}\right)^{\alpha}X_N,\\
X_N=&x(b)
\end{split}
\end{equation}
which gives an approximate solution of the Euler-Lagrange equation (\ref{ELH}).

\end{appendix}

\vskip 5mm
Jacky Cresson (*) and Anna Szafra\'{n}ska (**)

\begin{small}
(*)  Laboratoire de Math{\'e}matiques Appliquées de Pau, UMR CNRS 5142,

Université de Pau et des Pays de l'Adour,

avenue de l'Université, BP 1155, 64013 Pau Cedex, France.
\vskip 1mm
(**) Department of Differential Equations and Mathematics Applications,

Gda\'nsk University of Technology,

G. Narutowicz Street 11/12, 80-233 Gda\'nsk, Poland

E-mail: aszafranska@mif.pg.gda.pl
\end{small}

\end{document}